\documentclass{amsart}

\usepackage{ifthen}
\usepackage{graphicx}
\usepackage{amssymb}
\usepackage{enumerate}
\usepackage{url}
\newtheorem{anyprop}{Anyprop}[section]

\newtheorem{theorem}[anyprop]{Theorem}
\newtheorem{lemma}[anyprop]{Lemma}
\newtheorem{proposition}[anyprop]{Proposition}
\newtheorem{corollary}[anyprop]{Corollary}

\theoremstyle{definition}

\newtheorem{remark}[anyprop]{Remark}

\theoremstyle{remark}

\numberwithin{equation}{section}

\usepackage{amscd}
\usepackage{amssymb}
\input xy
\xyoption{all}

\begin{document}
\title[TOWARDS AN HOMOLOGICAL GENERALIZATION OF THE DCT]
{TOWARDS AN HOMOLOGICAL GENERALIZATION OF THE DIRECT SUMMAND THEOREM}


\author[Juan D. Velez]{Juan D. V\'elez}
\author[Danny de Jes\'us G\'omez-Ram\'irez]{Danny de Jes\'us G\'omez-Ram\'irez}
\address{Universidad Nacional de Colombia, Escuela de Matem\'aticas,
Calle 59A No 63 - 20, N\'ucleo El Volador, Medell\'in, Colombia.}
\address{Vienna University of Technology, Institute of Discrete Mathematics and Geometry,
wiedner Hauptstaße 8-10, 1040, Vienna, Austria.}
\email{jdvelez@unal.edu.co}
\email{daj.gomezramirez@gmail.com}

\begin{abstract}
We present a more general (parametric-) homological characterization
of the Direct Summand Theorem. Specifically, we state 
two new conjectures: the Socle-Parameter conjecture (SPC) in its
weak and strong forms. We give a proof for the week form by showing that it
is equivalent to the Direct Summand Conjecture (DSC), now known to be true
after the work of Y. Andr\'{e}, based on Scholze's theory of perfectoids. 
Furthermore, we prove the SPC in its strong form for the
case when the multiplicity of the parameters is smaller or equal than two.
Finally, we present a new proof of the DSC in the
equicharacteristic case, based on the techniques thus developed.


\end{abstract}
\maketitle

\smallskip


\section{Introduction}

Fundamental work of Hochster, Peskine, Szpiro, and Serre during the second
half of the twentieth century allowed to develop a theory of multiplicities,
along with the introduction of powerful prime characteristic methods in
commutative algebra \cite{hochsterhomological}, \cite{peskineszpiro}, \cite%
{serrelocal}. As a by-product of this effort for understanding general
properties of commutative rings and their prime spectra in terms of
homological and algebraic invariants, a collection of \emph{homological
conjectures} emerged \cite{hochsterhomological}. Further research of
Hochster's showed that most of these conjectures were essentially new ways
of describing a quite central algebraic splitting phenomenon for a
particular kind of ring extensions: most of these homological questions
turned out to be equivalent to the \emph{Direct Summand Conjecture} (DSC)  
\cite{hochsterhomological}, \cite{hochstercanonical}, \cite{ohidirsumcon}, 
\cite{robertsdirsumcon}, \cite{mcculloughthesis}.

Specifically, the DSC states that if $R\hookrightarrow S$ is a finite
extension of rings and if $R$ is a regular ring, then $R$ is a direct
summand of $S$ as $R$-module. Or equivalently, this extension splits as a
map of $R$-modules, i.e., there exists a retraction $\rho :S\rightarrow R$
sending $1_{R}$ to $1_{S}$.

After many decades, the Direct Summand Conjecture (or Direct Summand Theorem) was finally proved in the
former general form by Y. Andr\'{e}, by first reducing to the case of
unramified complete regular local rings, as it had been suggested by the
seminal work of Hochster's, and then by using Scholze's theory of
perfectoids \cite{andre}, \cite{bhatt}. Even though this fundamental
question was settled, the former results (regarded as a whole) suggest some
directions for future research in local homological algebra. 

In this work we present an equivalent form of the DSC given in terms of an
estimate for the difference of the lengths of the first two Koszul's
homology groups of quotients of Gorenstein rings by principal zero-divisor
ideals. S. P. Dutta and P. Griffith (see \cite[Theorem 1.5]{duttagriffith})
had obtained similar results in an independent manner, although in a rather
different context (for the case of complete and almost complete
intersections).

In Sections 3 and 4 we state an equivalent form of the DSC conjecture in
terms of the existence of annihilators of zero divisors on Gorenstein local
rings not belonging to parameters ideals \cite{velezflorez}, \cite%
{junesthesis}. Based on these results, we find a new conjecture equivalent,
in its weak form, to the DSC (\S 6). In its strong form, this conjecture
states that if $(T,\eta )$ is a Gorenstein local ring of dimension $d$ and $%
\{x_{1},...,x_{d}\}\subseteq T$ is any system of parameters, and if we
denote by $Q$ the ideal generated by these parameters, then for any zero
divisor $z\in T,$ and for any lifting $u\in T$ of a socle element in $T/Q$
(i.e. $\mathrm{Ann}_{T/Q}(\bar{\eta})=(\bar{u})$) it must hold that $uz\in
Q(z).$ 

This rather technical condition allows for more flexibility when one tries
to do computations in particular examples (see for instance the proof of
Proposition \ref{15}). We called this conjecture \emph{the Socle-Parameters
Conjecture (Strong Form) or SPCS }for short. We obtain the weak form if we
add the requirement that in the mixed-characteristic case $\mathrm{char}%
T/\eta =p>0$, and $x_{1}=p$. The SPCS is at the same time equivalent to a
very general and homological condition involving the lengths of the Koszul's
homology groups: 
\begin{equation*}
\ell (H_{0}(\underline{x},T/(z)))-\ell (H_{1}(\underline{x},T/(z)))>0.
\end{equation*}%
Evidently, This last condition only involves homological estimates. Then, a
theorem of Ikeda (see \cite[Corollary 1.4]{hunekeremarkmulti}) helps one to
show that the SPCS is true when the multiplicities of the parameters are
smaller or equal than two, suggesting a possible induction as a way of
solving (\S 8).This approach has its origins in the work of J. D. V\'{e}lez.
Specifically, in the idea of proving the DSC by means of annihilators, an
idea first introduced in his thesis (see \cite[Lemma 3.1.2.]{velezthesis}
and \cite{velezsplitting}). The reduction to the case where $S=T/J$, $T$ is
a Gorenstein local ring, and $J$ is a principal ideal, was stated in a
private communication from J. D. V\'{e}lez to M. Hochster, in 1996, and
appears more explicitly in \cite{velezflorez}, and in \cite{junesthesis}.
Similar results related to the SPCS were obtained independently by J.
Strooker and J. St\"{u}ckrad \cite{strookerstueckrad}. 

\section{Preliminary results}

Let $(R,m)\hookrightarrow S$ be a module finite extension. Let $%
s_{1},...,s_{n}\in S$ be generators of $S$ as an $R$-module. Then, there
exist monic polynomials $f_{i}(y_{i})\in R[y_{i}]$ such that $%
f_{i}(y_{i})=y_{i}^{m_{i}}+a_{i1}y_{i}^{n_{i}-1}+\dots +a_{im_{i}}$, $%
a_{ij}\in R$ and $f_{i}(y_{i})=0$. One can define a homomorphism of $R$%
-algebras from 
\begin{equation*}
T=R[y_{1},...,y_{n}]/(f_{1}(y_{1}),...,f_{n}(y_{n}))
\end{equation*}%
to $S$ sending each $y_{i}$ to $s_{i}$. If $J$ denotes the kernel of this
map, $S\cong T/J$. Finally, since $\mathrm{dim}T=\mathrm{dim}R=\mathrm{dim}%
T/J$, by reason of the finiteness of the extension, $J$ should be contain in
a minimal prime ideal, that means, $\mathrm{ht}J=0$. Later we will develop
all the necessary facts in order to prove that, if the residue field is
algebraically closed, then we can reduce to the case where $a_{ij}\in m$.

\begin{remark}
\label{2}  Let $R \hookrightarrow S$ be a finite extension of Noetherian
rings, where $(R,m)$ is local. Then the maximal spectrum of $S$, $\mathrm{Spec}_mS$, 
is equal to $V(mS)\subseteq \mathrm{Spec} S$. In fact, since $%
R/m\hookrightarrow S/mS$ is finite, $\mathrm{dim}S/mS=\mathrm{dim}R/m=0$.
Therefore  $S/mS$ is Artinian. Hence $\mathrm{Spec}_mS/mS=\mathrm{Spec} S/mS$ is
finite. Now, let $\eta\in \mathrm{Spec}_mS$ then $\mathrm{dim}R/(R\cap\eta)=%
\mathrm{dim}S/\eta=0$. Hence $R/R\cap\eta$ is a field (a domain of dimension
zero), and $R\cap\eta=m$. Consequently $\mathrm{Spec}_mS=V(mS)$.
\end{remark}

\begin{lemma}
Let $(R,m,k)$ be a local complete ring and $R\hookrightarrow S$ a finite
extension. Assume that $\mathrm{Spec} S=V(mS)=\{\eta_1,...,\eta_n\}$. Then  $S$
is naturally isomorphic, as a ring, to $S_{\eta_1}\times\cdots\times
S_{\eta_n}$.
\end{lemma}

\begin{proof}
By Remark \ref{2} and the comments made at the beginning of this section, $%
S/mS=(S/mS)_{\eta_1}\times\ldots\times(S/mS)_{\eta_n}$. But this is
equivalent to the existence  of idempotent orthogonal elements $%
e_1,...,e_n\in S/mS$, which means that, $e_i^2=e_i$, $\sum_{i=1}^ne_i=1$ and 
$e_ie_j=0$ for all $i\neq j$. In fact  $e_i\notin\eta_i$ and $e_i\in\eta_j$
for all $i\neq j$. Let $p(t)=t^2-t\in S[t]$. Then $\overline{p}%
(t)=(t-e_i)(t-(1-e_i))\in (S/mS)[t]$ for all $i$, and $(t-e_i,t-(1-e_i))=(1)$%
, because $(t-e_i)-(t-(1-e_i))=1-2e_i$ and $%
e_i(t-e_i)-e_i(t-(1-e_i))=-e_i^2=-e_i$, then $1\in (t-e_i,t-(1-e_i))$. By
Hensel's Lemma (see \cite[Theorem 7.18.]{eisenbud}) there exist linear monic
polynomials $t-E_i,t-D_i\in S[t]$ such that $t-\overline{E_i}=t-e_i$ and $t-%
\overline{D_i}=t-(1-e_i)$ in $(S/mS)[t]$ (i.e. $\overline{E_i}=e_i$ and $%
\overline{D_i}=1-e_i$) and  $p(t)=(t-E_i)(t-D_i)$. Let's fix such $E_i\in R$
for $i=1,...,n-1$. Then $p(E_i)=E_i^2-E_i=0$, that is, $E_i^2=E_i$. Besides,
for $i\neq j$, $(E_iE_j)^n=E_i^nE_j^n=E_iE_j\in mS$. Thus, for any maximal
ideal of $S$, say $\eta_r$, $E_iE_j/1\in
\cap_{m=1}^n(\eta_iS_{\eta_i})^n\subseteq S_{\eta_i}$, because $mS\subseteq
\eta_i$ by Remark \ref{2}. Therefore, by Krull's Intersection Theorem $%
\cap_{m=1}^n(\eta_iS_{\eta_i})^n=(0)$ (see \cite[Corollary 5.4.]{eisenbud}).
Hence, there exists $s\notin\eta_i$ such that $sE_iE_j=0$, which means that $%
(0:E_iE_j)\nsubseteq\eta_r$ and that holds for all maximal ideals $\eta_r$.
In conclusion, $(0:E_iE_j)=S$, so $E_iE_j=0$. Define $E_n=1-%
\sum_{i=1}^{n-1}E_i$. Then 
\begin{equation*}
E_n^2=1-2(\sum_{i=1}^{n-1}E_i)+(\sum_{i=1}^{n-1}E_i)^2=1-2(%
\sum_{i=1}^{n-1}E_i)+\sum_{i=1}^{n-1}E_i^2=1-\sum_{i=1}^{n-1}E_i=E_n,
\end{equation*}
and 
\begin{equation*}
E_jE_n=E_i-\sum_{i=1}^{n-1}E_jE_i=E_j-E_j^2=0,
\end{equation*}
for all $j<n$, and $\overline{E_n}=1-\sum_{i=1}^{n-1}\overline{E_i}%
=1-\sum_{i=1}^{n-1}e_i=e_n$. Therefore $\{E_1,...,E_n\}$ is a set of
idempotent orthogonal elements for $S$. \newline
\indent
Now, $E_i\in \cap _{i\neq j}\eta_i\smallsetminus\eta_j$, because $\overline{%
E_i}=e_i\in \cap_{i\neq j} \overline{\eta_i}\smallsetminus\ \overline{\eta_j}
$. So, $S_{E_i}=S[e_i^{-1}]$ is a local ring with maximal ideal $\eta_i^e$
because $\eta_i$  is the only maximal ideal not containing it. Then $%
S[e_i]\cong S_{\eta_i}$. Furthermore, there are natural homomorphism of
rings $\alpha:S\rightarrow SE_i$ sending $E_i\rightarrow E_i^2=E_i$. So $%
\alpha$ send $E_i$ to the unity on $S{E_i}$, and then $\alpha$ induces a
homomorphism from $S_{E_I}$ to $SE_i$ which is clearly bijective. Thus, $%
S_{E_I}\cong SE_i$. In conclusion, there is a natural isomorphism as follows 
$S\cong \oplus_{i=1}^nSE_i \cong
\oplus_{i=1}^nS_{E_i}\cong\oplus_{i=1}^nS_{\eta_i}$.
\end{proof}

\begin{corollary}
\label{3}  Let $(R,m,k)$ be a local ring with algebraically closed field $k$
and $B=R[x]/(F(x))$, where $F(x)$ is a monic polynomial of degree $n$. Then
there  exist monic polynomials $G_i(x)$ of degree $n_i$ such that $%
\sum_{i=1}^{r}n_i=n$, $B=\oplus_{i=1}^{r} R[x]/(G_i(x))$ as rings, and 
\begin{equation*}
G_i(x)=x^{n_i}+a_{i1}x^{n_i-1}+\cdots+a_{in_i},
\end{equation*}
where all $a_{ij}\in m$.
\end{corollary}

\begin{proof}
Since $k$ is algebraically closed we can factor $f(x):=\overline{F}%
(x)=\prod_{i=1}^{r}(x-b_{i})^{n_{i}}\in k[x]$, for some $b_{i}\in k$ and $%
n_{i}\in \mathbb{N}$, with $\sum_{i=1}^{r}n_{i}=n$. Besides, $B/mB\cong
k[x]/(f(t))$ is an Artinian ring with maximal ideal $\overline{\eta _{i}}%
=(x-b_{i})$ for $i=1,...,r$, therefore $B/mB\cong \oplus
_{i=1}^{r}k[x]/((x-b_{i})^{n_{i}})$ (see \cite[Theorem 8.7 and proof]{atimac}%
). Now, by Hensel's Lemma there exist monic polynomials $F_{i}(x)\in R[x]$
such that $F[x]=\prod_{n=1}^{r}F_{i}(x)$ and $\overline{F_{i}}%
(x)=(x-b_{i})^{n_{i}}$. By Remark \ref{2}, $\mathrm{Spec}_{m}B=\{\eta
_{1},...,\eta _{r}\}=V(mB)$. \newline
\indent Now, let $B_{i}\in R$ such that $\overline{B_{i}}=b_{i}$. Since $%
\eta _{i}=(x-b_{i})$ in $B/mB$ then we see by the correspondence between the
ideals of $B/mB$ an the ideals of $B$ containing $mB$ that $\eta
_{i}=(x-B_{i})+mB$. Now, $(F(x))R[x]_{\eta _{i}}=(F_{i}(x))R[x]_{\eta _{i}}$%
, because $(F_{j}(x))+\eta _{i}=R[x]$ for all $i\neq j$, since mod $mR[x]$
this ideal corresponds to $((x-b_{i})^{n_{j}})+(x-b_{i})=\mathrm{rad}%
(x-b_{i},x-b_{j})=k[x]$ for all $i\neq j$. Therefore $F_{j}(x)$ is a unit in 
$R[x]_{\eta _{i}}$. Besides, the ring $R[x]/(F_{i}(x))$ is local with
maximal ideal $(x-B_{i})+m^{e}$, (which we denote again by $\eta _{i}$).
This is because any maximal ideal should contain the expansion of $m$ (the
extension $R\hookrightarrow R[x]/(F_{i}(x))$ is finite) and this ring module 
$m^{e}$, is the local ring $(k[x]/((x-b_{i})^{n_{i}}),(x-b_{i})^{e})$. So,%
\begin{equation*}
B_{\eta _{i}}=R[x]_{\eta _{i}}/(F(x))R[x]_{\eta _{i}}\cong R[x]_{\eta
_{i}}/F_{i}(x)R[x]_{\eta _{i}}
\end{equation*}%
\begin{equation*}
\cong (R[x]/(F_{i}(x)))_{\eta _{i}}\cong R[x]/(F_{i}(x)).
\end{equation*}%
Thus, by the previous lemma, $B\cong \oplus _{i=1}^{r}B_{\eta _{i}}\cong
\oplus _{i=1}^{r}R[x]/(F_{i}(x))$. \newline
\indent Finally, the isomorphism of rings $\theta _{i}:R[x]\rightarrow R[t]$%
, sending $x\mapsto t+B_{i}$, induces an isomorphism of rings $\overline{%
\theta _{i}}:R[x]/(F_{i}(x))\rightarrow R[t]/(F_{i}(t+B_{i})$. But the
reduction of $F_{i}(x)$ mod $mR[x]$ is exactly $%
((t+B_{i})-B_{i})^{n_{i}}=t^{n_{i}}$, because translation and reduction mod $%
m$ commutes. But that means exactly that $%
G_{i}(t)=t^{n_{i}}+a_{i,1}t^{n_{i}-1}+\cdots +a_{i,n_{i}}\in R[t]$ with $%
a_{i,j}\in m$, for any indices $i$, $j$. In conclusion, we get an
isomorphism of rings between $B$ and $\oplus _{i=1}^{r}R[x]/(G_{i}(t))$
satisfying the conditions of our corollary.
\end{proof}

\begin{remark}
\label{4}  Let $i:R\hookrightarrow S=S_1\times...\times S_n$ be an extension
of rings, where $R$ is and integral domain. Then there exists a $S_i$ such
that $\pi_i\circ i:R\hookrightarrow S_i$ is also an extension, where $\pi_i$
is the natural projection. Suppose by  contradiction that for any $i$ there
exist $a_i\neq0\in R$ such that $\pi_i(i(a_i))=0$. Then, if $%
a=\prod_{i=1}^na_i\neq0$, for any $j$, 
\begin{equation*}
\pi_j(a_i)=\prod_{i=1}^n\pi_j(i(a_i))=\pi_j(i(a_j))\prod_{i\neq
j}\pi_j(i(a_i))=0.
\end{equation*}
Therefore $i(a)=(\pi_1(i(a)),..., \pi_n(i(a))=0$, contradicting our
hypothesis.
\end{remark}

\begin{theorem}
\label{5}  Let $(R,m,k)$ be a regular complete local ring with algebraically
closed residue field $k$. Then, to prove the DSC for $R$ it is enough to
consider finite extensions $R\hookrightarrow S$, where $S=T/J$, 
\begin{equation*}
T=R[y_1,...,y_r]/(f_1(y_1),...,f_r(y_r)),
\end{equation*}
$f_i(y_i)=y_i^{n_i}+a_{i,1}y_1^{n_i-1}+\cdots+a_{i,n_i}$ with $a_{i,j}\in m$
and $J\subseteq T$ is an ideal of height zero.
\end{theorem}

\begin{proof}
Let us fix a finite extension $R\hookrightarrow S$. By the discussion above,
we know that $S=T/J$, where $T=R[y_1,...,y_r]/(f_1(y_1),...,f_r(y_r))$ (the
coefficients of the monic polynomials $f_i(y_i)$ are not necessarily in $m$%
), and $\mathrm{ht}(J)=0$. It is elementary to see that 
\begin{equation*}
T\cong\otimes_{i=1}^rR[y_i]/(f_i(y_i)).
\end{equation*}
Write $B_i=R[y_i]/(f_i(y_i))$ then by Corollary \ref{3} $B_i\cong\oplus_{%
\alpha=1}^{m_i}R[y_i]/(f_{i\alpha}(y_i))$ with $f_{i\alpha}(y_i)=y_i^{n_{i%
\alpha}}+a_{i\alpha1}y_i^{n_{i\alpha}-1}+\cdots+ a_{i\alpha n_{i\alpha}}$
and $a_{i\alpha j}\in m$. Furthermore, by the distributive law between
tensor products and direct sums (see \cite{matcomrintheo}) we get  
\begin{equation*}
T\cong
\otimes_{i=1}^r(\oplus_{\alpha=1}^{m_i}R[y_i]/(f_{i\alpha}(y_i)))\cong%
\oplus_w(\otimes_{i=1}^rR[y_i]/(f_{i\alpha_i} (y_i)))\cong\oplus_wT_w,
\end{equation*}
where $w=(\alpha_1,...,\alpha_r)$, $1\leq \alpha_i\leq m_i$  and 
\begin{equation*}
T_w\cong\otimes_{i=1}^rR[y_i]/(f_{i\alpha_i}(y_i))\cong
R[y_i,...,y_r]/(f_{1\alpha_1}(y_1),...,f_{r\alpha_r}(y_r))
\end{equation*}
where each $f_{i\alpha_i}(y_i)$ has lower coefficients in $m$, as desired.
Besides, $J=\oplus_wJ_w$ and then $S\cong\oplus_wT_w/J_w$. \newline
\indent
Finally, by Remark \ref{4}, there is an $\alpha$  such that $%
R\hookrightarrow T_w/J_w$ is an extension. But if there is a retraction  $%
\rho_w:T_w/J_w\hookrightarrow R$, then $\rho=\pi_w\circ \rho_w:S\rightarrow R
$ is also a retraction. Besides $ht(J_w)=0$, because $\mathrm{dim}T_{ w}=%
\mathrm{dim}R=\mathrm{dim} T_w/J_w$. In conclusion, $S_w=T_w/J_w$ has the
desired form of our proposition and then it is enough to prove the DSC in
this case.
\end{proof}

\begin{proposition}
\label{6}  Let $(R,m,k)$ be a regular local ring of dimension $d$, and 
\begin{equation*}
T=R[y_1,...,y_r]/(f_1(y_1),...,f_r(y_r)),
\end{equation*}
$f_i(y_i)= y_i^{n_i}+a_{i,1}y_1^{n_i-1}+\cdots+a_{i,n_i}$, with $a_{i,j}\in m
$. Then $T$ is a Gorenstein local ring with maximal ideal $\eta=m+(\overline{%
y_1},...,\overline{y_r})$.
\end{proposition}

\begin{proof}
First we see that $T$ is a local C-M ring. In fact, let $m_1$ be any maximal
ideal of $T$. Then $m_1\cap R=m$, because $\mathrm{dim}(R/(m_1\cap R))=%
\mathrm{dim}(T/m_1)=0$ and $m_1\cap R\in \mathrm{Spec}(R)$. Therefore $%
R/(m_1\cap R)$ is a field. Thus, $y_i^{n_i}\in m_1$ and therefore $y_i\in m_1
$. In conclusion, $m_1=\eta$. Now, $R[y_1,...,y_r]$ is a C-M ring (see \cite[%
Proposition 18.9]{eisenbud}). Thus $B=R[y_1,...,y_r]_\eta$ is also C-M. Let $%
m=(x_1,...,x_d)$, where $d=\mathrm{dim}R$. Then $%
\{f_1(y_1),...,f_r(y_r),x_1,...,x_d\}\subseteq B$ is a system of parameters
because $\mathrm{dim}B=ht(\eta)=\mathrm{dim}R[y_1,...,y_r]=\mathrm{dim}%
R+r=d+r$, since $R[y_1,...y_r]$ is C-M then by previous comments is
equidimensional, and $\mathrm{rad}(f_1,...,f_r,x_1,...,x_d)=\eta B$. Then $%
\{f_1,...,f_r,x_1,...,x_d\}$ is a regular sequence in $B$. Thus, $%
B/(f_1,...,f_r)\cong T$ is C-M, due to the fact that $\mathrm{dim}%
B/(f_1,...,f_r)=d$, and $\{\overline{x_1},...,\overline{x_d}\}\subseteq
B/(f_1,...,f_r)$ is a regular sequence. \newline
\indent
Finally, let $Q=(y_1,...,y_r,x_1,...,x_d)$ be the ideal generated by the
system of parameters $\{y_1,...,y_r,x_1,...,x_d\}$. Then  $T/Q\cong
k[y_1,...,y_r]/(y_1^{n_1},...,y_r^{n_r})=k[w_1,...,w_r]$, where $w_i=%
\overline{y_i}$. Let us see that 
\begin{equation*}
\mathrm{Ann}_{T/Q}((w_1,...,w_r))=(\prod_{i=1}^rw_i^{n_i-1}).
\end{equation*}
In fact, if $\overline{h}\in \mathrm{Ann}_{T/Q}((w_1,...,w_r))$ and $%
c\prod_{i=1}^rw_i^{m_i}\neq0$  is a monomial of $\overline{h}$ such that
there exists $m_j\in \mathbb{N}$ with $m_j<n_j-1$, then $w_j\overline{h}%
\neq0\in T/Q$, because the monomial term $w_jc\prod_{i=1}^rw_i^{m_i}$ has
the power $m_j+1<n_j$ on $w_j$, which is a contradiction, by the reason that 
$h$ is on the socle. Therefore, $m_i\geq n_i-1$ for all $i$ and so $%
\overline{h}\in (\prod_{i=1}^rw_i^{m_i})$. The other contention is clear,
and then the socle has dimension one. In conclusion, $(T,\eta)$ is a
Gorenstein local ring.
\end{proof}

\section{The DSC in terms of Annihilators}

Now we make preparations for the proof of the following fact: let $%
h:(R,m)\rightarrow (T,\eta)$ be a finite homomorphism of local rings, i.e. $%
h(m)\subseteq \eta$  where $T$ is a local $R-$free  ring, with $T/mT$
Gorenstein, and let $S=T/J$, for some ideal $J\subseteq T$. Then $%
h:R\rightarrow S$ splits if and only if $\mathrm{Ann}_TJ\nsubseteq mT$ (by
abuse of notation we denote by $h$ again its composition with the natural
projection $\pi:T\rightarrow S$).

\begin{proposition}
\label{7}  Let $(A,\eta,k)$ be a local Gorenstein ring of dimension zero
(i.e. $\mathrm{dim}_k(\mathrm{Ann}_A\eta)=1$). Let $u\in\eta$ such that $(u)=%
\mathrm{Ann}_A\eta$.  Then $u\in I$ for any ideal $I\neq (0)\subseteq A$.
\end{proposition}

\begin{proof}
Clearly, we can assume that $I\subseteq\eta$. We know $\mathrm{nil}(A)=\eta$%
, therefore there exists $n\in \mathbb{N}$ such that $\eta^n=0$. Let $x\neq
0\in I$. Then $\eta^{n-1}x\subseteq\eta^n=(0)$. Let $r\in \mathbb{N}$ be
such that $\eta^rx=(0)$ but $\eta^{r-1}x\neq(0)$. Hence $\eta(%
\eta^{r-1}x)=(0)$ and so $\eta^{r-1}x\subseteq \mathrm{Ann}_A\eta=(u)$. Then 
$(u)$ contains a nonzero element of the form $bx$, where $b\in\eta^{r-1}$.
That means that there exists $c\in A\smallsetminus\eta$  with $cu=xb$, so $%
u=(c^{-1}b)x\in I$.
\end{proof}

\begin{remark}
If $h:R\hookrightarrow T$ is any homomorphism of rings, we can consider $T^*=%
\mathrm{Hom}_R(T,R)$ as a $T-$module with the following action:  fix $t\in T$
and define $(t\cdot\phi)(x):=\phi(tx)$, for $\phi\in \mathrm{Hom}_R(T,R)$ and $%
x\in T$.
\end{remark}

\begin{remark}
\label{8}  Let $(R,m)$ be a local ring, $T$ a finitely generated $R-$free
module, and $\theta:T\rightarrow T$ an $R-$homomorphism. Then $\theta$ is an
isomorphism of $R-$modules if and only if $\overline{\theta}:T/mT\rightarrow
T/mT$ is an isomorphism of $K-$vector  spaces. In fact, if $A\in M_{n\times
n}(R)$ is the matrix defining $\theta$, then $\theta$ is an isomorphism if
and only if $\det A$ is an unit, which means that $\det A\notin m$. But that
is equivalent to saying that $det\overline{A}\neq 0\in k$, where $\overline{A%
}$ is the reduction  of $A$ mod $m$. Finally, since $\overline{A}$ is the
matrix defining $\overline{\theta}$, the last condition is equivalent to
saying that $\overline{\theta}$ is an isomorphism of $k-$vector spaces.
\end{remark}

\begin{theorem}
\label{9} Let $(R,m)$ and $(T,\eta)$ be local rings. Assume that $T/mT$ is
Gorenstein. Let $h:R\rightarrow T$ be a finite homomorphism of local rings,
such that $T$ is $R-$free. Then there exists a $T-$isomorphism $%
\beta:T\rightarrow T^*$ such that for any ideal $J\subseteq T$, $%
\beta^{-1}((T/J)^*)=\mathrm{Ann}_JT$.
\end{theorem}

\begin{proof}
We identify $(T/J)^*$ with $\{f\in T^*:f(J)=0\}$. We know that $\mathrm{dim}%
T/mT=\mathrm{dim}R/m=0$, since $T/mT$ is a finitely generated $R/m-$module. 
So, fix $u_1\in\eta$ such that $(\overline{u_1})=\mathrm{Ann}_{T/m}\overline{%
\eta}$,  since $T/mT$ is Gorenstein. Let $\{u_2,...,u_d\}\subseteq T$ such
that $T/mT=(\overline{u_1},...,\overline{u_d})$ (where $T\cong R^d$). Then,
by the Lemma of Nakayama $T=(u_1,...,u_d)$, $T$ is generated by $u_1,...,u_d$
as an $R-$free module. In fact, let $\{w_1,...,w_d\}\subseteq T$ be an $R-$%
basis for $T$. Define $\theta:T\rightarrow T$ by  $w_i\rightarrow u_i$, then
the induced $\overline{\theta}:T/mT\rightarrow T/mT$ is clearly an
isomorphism of $k-$vector spaces. Since $T$ is $R-$free, by Remark \ref{8}, $%
\theta$ is an isomorphism which means just that $\{u_1,...,u_d\}\subseteq T$
is an $R-$basis for $T$. Let $u_1^*\in T^*$ be the dual element and define $%
\beta:T\rightarrow T^*$ by $t\rightarrow t\cdot u_1^*$, where $(t\cdot
u_1^*)(t_1):=u_1^*(tt_1)$,  for all $t_1\in T$. By definition, it is clear
that $\beta$ is an $T-$homomorphism.  Now, we can make the natural
identifications $T^*=\mathrm{Hom}_R(R^d,R)\cong(\mathrm{Hom}_r(R,R))^d=R^d$.
Therefore, $T^*$ is an $R-$free  module of dimension $d$ and $\beta$ is an $%
R-$isomorphism if and only if $\overline{\beta}:T/mT\rightarrow T^*/mT^*$ is
so, due to Remark \ref{8} ($T\cong T^*$ as $R-$free modules). But $\overline{%
\beta}$ is an isomorphism of $k-$vector spaces if it is injective.  Suppose
by contradiction that $\mathrm{ker}\overline{\beta}\neq(0)$. Then, by
Proposition \ref{7}, $\overline{u_1}\in \mathrm{ker}\overline{\beta}$.  That
means,  $\beta(u_1)\in mT^*$, so there exist $m_1,...,m_d\in m$ such that $%
\beta(u_1)=u_1u_1^*=\sum_{i=1}^dm_iu_i^*$. But, evaluating at $1$ we get: $%
1=u_1u_1^*(1)=\sum_{i=1}^dm_iu_i^*(1)\in m$, a contradiction. In conclusion,
$\beta$ is a $T-$isomorphism. \newline
\indent
For the last part, let $a\in \mathrm{Ann}_TJ$. Then, for any $j\in J$, $%
\beta(a)(j):=(au_1^*)(j)=u_1^*(aj)=u_1^*(0)=0$.  Therefore, $\beta(a)\in
(T/J)^*$ and thus $a=\beta^{-1}(\beta(a))\in\beta^{-1}((T/J)^*)$. \newline
\indent
Conversely, take $\phi_0\in \beta^{-1}((T/J)^*)$. Then there exists a $%
\phi\in (T/J)^*$ such that $\phi_0=\beta^{-1}(\phi)$, i.e. $\phi=\phi_0u_1^*$
and $\phi(J)=0$, which means that $u_1^*(\phi_0j)=0$ for all $j\in J$. Now,
let's fix $j_0\in J$. Then, for all $t\in T$, $\beta(\phi_0j_0)(t)=u_1^*(%
\phi_0j_0t) =u_1^*(\phi_0(j_0t))=0$, because $j_0t\in J$. Therefore $%
\beta(\phi_0j_0)\equiv 0$ and then, by the first part $\phi_0j_0=0$, which
means that $\phi_0\in \mathrm{Ann}_TJ$.
\end{proof}

\begin{theorem}
\label{10}  Let $h:R\rightarrow T/J$ be a finite homomorphism of local rings
such that $(T,\eta)$ is a local $R-$free ring,  with $T/mT$ Gorenstein, and $%
J\subseteq T$ an ideal. Assume that the structure of $R-$module of $T/J$
inherited by the $R-$structure of  $T$ is the same as the one induced by $h$%
. Then $R\hookrightarrow T/J$ splits if and only if $\mathrm{Ann}%
_TJ\nsubseteq mT$.
\end{theorem}

\begin{proof}
Assume that $\rho: T/J\hookrightarrow R$ is a splitting $R-$homomorphism.
Then $\rho\in (T/J)^*$ and $\rho(1)=1$. By the last theorem, there  exists $%
\rho_0\in \mathrm{Ann}_TJ$ such that $\rho_0=\beta^{-1}(\rho)$ which means,
in particular, that $1=\rho(1)=\rho_0\cdot u_1^*(1)=u_i^*(\rho_0)\notin m$.
Then $\rho_0\notin m T$, so $\mathrm{Ann}_TJ\nsubseteq mT$. \newline
\indent
Conversely, take $a\in \mathrm{Ann}_TJ\smallsetminus mT$. Then, by
Proposition \ref{7}, $\overline{u_1}\in(\overline{a})\subseteq T/mT$, which
means that there exists $m_i\in m$, with $u_1-a=\sum_{i=1}^dm_iu_i$, where $%
\{u_1,...,u_d\}\subseteq T$ is an $R-$basis for $T$ as in the last Theorem.
Then, 
\begin{equation*}
au_i^*(1)=(1-m_1)u_1u_1^*(1)-\sum_{i=2}^dm_iu_iu_i^*(1)
\end{equation*}
\begin{equation*}
=(1-m_1)-\sum_{i=2}^dm_iu_1^*(u_i)=(1-m_1)+0=1-m_1\notin m.
\end{equation*}
So $\rho=\beta((1-m_i)^{-1}a)$ satisfies that $\rho\in (T/J)^*$ because $%
\beta^{-1}(\rho)=(1-m_1)^{-1}a\in \mathrm{Ann}_TJ$, and by Theorem \ref{9} $%
\beta^{-1}((T/J)^*)=\mathrm{Ann}_TJ$. Besides, $%
\rho_1=(1-m_i)^{-1}au_1^*(1)=(1-m_1)^{-1}(1-m_1)=1$, which implies that $%
\rho:T/J\rightarrow R$  is the desired splitting $R-$homomorphism.
\end{proof}

\section{Reduction to the case where J is principal}

In the next proposition we will prove that we can reduce to the case where $J
$ is a principal ideal generated by an element in $mT$.

\begin{proposition}
\label{11}  Let $(R_0,m_0,k_0)$ be a regular local ring and $%
R_0\hookrightarrow T_0$ be a finite extension, where 
\begin{equation*}
T_0=R_0[y_1,...,y_r]/(f_1(y_1),...,f_r(y_r)),
\end{equation*}
and each $f_i(y_i)=y_i^{n_i}+a_{i1}y_i^{n_i-1}+\cdots+a_{in_i}$, with $%
a_{ij}\in m$, for all indices $i$,$j$. Let $S=T_0/J$, with $%
J=(g_1,...,g_s)\subseteq T_0$ such that  $J\cap R_0=(0)$. Let $x_1,...,x_s$
be new variables, and let $R$ be $R_0[x_1,...,x_s]$,and let $m$ be $%
m_0+(x_1,...,x_s)$, a maximal ideal of $R$.  Let $g$ be the element $%
x_1g_1+\cdots+x_sg_s$. Write 
\begin{equation*}
T=R\otimes_{R_0}T_0\cong R[y_1,...,y_r]/(f_1(y_1),...,f_r(y_r)).
\end{equation*}
Then:

\begin{enumerate}
\item $(R_m,mR_m,k_0)$ is a regular local ring.

\item $(g)T_m\cap R_m=(0)$.

\item $R_m\hookrightarrow (T/(g))_m\cong T_m/(g)^e$ is a finite extension,
and $g\in (mR_m)T_m$.

\item $R_m\hookrightarrow T_m/(g)$ splits if and only if $R_0\hookrightarrow
T_0/J$ splits.
\end{enumerate}
\end{proposition}

\begin{proof}
(1) In general, if $R$ is regular then so is the polynomial ring $R[T]$ (see 
\cite{matcomrintheo}). In particular, $R_m$ is a regular local ring and $%
R_m/mR_m\cong R/m \cong R_0/m_0=k_0$.

(2) $R_0\hookrightarrow R$ is an $R-$free extension, then, in particular, it
is flat. Therefore, by tensoring $R_0\hookrightarrow T_0/J$ with $R$, we see
that $R\hookrightarrow R\otimes_{R_0}T_0/J\cong T/J^e$ is also an extension,
and since localization is flat too, we get an extension $R_m\hookrightarrow
T_m/J^e$. Because of $g\in J^e$, we get $gT_m\cap R_m=0$.

(3) Clearly, by definition $g\in(mR_m)T_m$. Now, by the previous paragraph $%
R_m\hookrightarrow T_m/(g)^e$ is an extension, and it is finite because $%
R_0\hookrightarrow T_0$ is  finite, in fact free. Then, after tensoring with 
$R_m$ we get a module finite extension $R_m\hookrightarrow T_m$, so $%
T_m/(g)^e$ is also a finitely generated $R_m-$module.

(4) Assume that $\rho_0:T_0/J\rightarrow R_0$ is an $R_0-$homomorphism such
that $\rho_0(1)=1$. Then, by tensoring with the flat $R_0-$module  $R_m$, we
get an $R_m-$homomorphism $\rho:T_m/J\hookrightarrow R_m$, with $\rho(1)=1$.
Now, composing $\rho$ with the natural map $T_m/(g)\rightarrow T_m/J$, we
obtain a retraction from $T_m/(g)$ to $R_m$.

Conversely, it is clear that $T_m$ satisfies the hypothesis of Proposition %
\ref{6}. Therefore, by Theorem \ref{10}, $\mathrm{Ann}_{T_m}(g)%
\nsubseteq(mT_m$. So let's choose 
\begin{equation*}
w=\sum_{\alpha}(h_{\alpha}(\underline{x})/k_{\alpha}(\underline{x}%
))y^{\alpha},
\end{equation*}
such that $wg=0$, where  $h_{\alpha}\in R$ and $k_{\alpha}\in
R\smallsetminus m$ (which is equivalent to saying that $k_{\alpha}(%
\underline{0})\notin m_0$). Here $y^{\alpha}$ denotes $y_1^{\alpha_1}...
y_{r}^{\alpha_r}$, $\alpha=(\alpha_1,...,\alpha_r)$, $0\leq \alpha_i<degf_i$
and some $h_{\beta}\notin m$ (that means exactly $w\notin mT_m$). We have $%
T=R\otimes_{R_0}T\cong T_0[x_1,...,x_s].$ Now, multiplying by the product of
the $k_{\alpha}(\underline{x})$, we can assume that $w=\sum_{\alpha}p_{%
\alpha}(\underline{x})y^{\alpha}\in T$, where some $p_{\alpha}\notin m$ and $%
0=wg=\sum_i \sum_{\alpha}p_{\alpha}(\underline{x})y^{\alpha} x_ig_i(%
\underline{y})$ in $T$. Now, the coefficient of $x_i$, which is zero in $T$,
is exactly $\sum_{\alpha}p_{\alpha}(\underline{0})y^{\alpha}g_i(\underline{y}%
)$, because the terms $y^{\alpha}g_i(\underline{y})$ are constants in $%
T=T_0[x_1,...,x_s]$. Therefore, if $w_0=\sum_{\alpha}p_{\alpha}(\underline{0}%
) y^{\alpha}\in T_0$, we have $w_0g_i=\sum_{\alpha}p_{\alpha}(\underline{0}%
)y^{\alpha}g_i(\underline{y})=0$, and thus $w_0\in \mathrm{Ann}_{t_0}J$. But 
$p_{\beta}(\underline{0})\notin m_0$, because $p_{\beta}(\underline{x}%
)\notin m$, so $w_0\notin m_0T_0$. In conclusion, $\mathrm{Ann}%
_{T_0}J\nsubseteq m_0T_0$,  which is equivalent by Theorem \ref{10} to the
fact that $R_0\hookrightarrow T_0/J$ splits.
\end{proof}

\section{The Socle-Parameter Conjecture}

In this, and in the next section, we state two new conjectures (The
Socle-Parameters Conjecture, in its strong form (SPCS), and in its weak form
SPCW), and we will prove that the SPCW is equivalent to the DSC, and that
the SPCS implies the SPCW. Besides, these two conjectures are equivalent in
the equicharacteristic case and therefore both are equivalent to the DSC in
the equicharacteristic case. 

However, as far as we know, the mixed characteristic case remains open. The
new approach shows that the DSC is, in essence, a problem concerning
algebraic and homological properties of Gorenstein local rings.

Let us start by reviewing some elementary notions. Let $R$ be an $\mathbb{N}$%
-graded ring such that $R_{0}$ is an Artinian ring, and such that $R$ is
finitely generated as an $R_{0}$-algebra. Let $M$ be a finitely generated $R$%
-module of dimension $d$. Then, it is elementary to see that each
homogeneous part $M_{n}$ is a finitely generated $R_{0}$-module and
therefore it has finite length (see \cite[Proposition 6.5.]{atimac}). It is
well known in this case that there exists a unique polynomial 
\begin{equation*}
p_{M}(t)=a_{d-1}t^{d-1}+\cdots +a_{0}\in \mathbb{Q}[t]
\end{equation*}%
of degree $d-1$, the \emph{Hilbert polynomial}, such that for $n\gg 0$, $%
p_{M}(n)=\ell (M_{n})$. The \emph{multiplicity} of $M$, $e(M)$, is defined
as $\ell (M)$ if $d=0$, and as $(d-1)!a_{d-1}$, if $d>0$ (see \cite[%
Definition 4.1.5.]{brunsherzog}. In particular, if $M$ has positive
dimension, then $e(M)>0$, since $a_{d-1}>0$. For the local case, assume that 
$(R,m)$ is a local ring, $M$ a finitely generated $R$-module of dimension $d,
$ and $I=(x_{1},...,x_{n})$ is an \emph{ideal of definition} of $M$. This
last condition means that $m^{r}M\subseteq IM$ for some $r>0$, which is
equivalent to saying that $x=x_{1},\cdots ,x_{n}$ is a \emph{multiplicity
system} on $I$ (i.e. $\ell (M/(x_{1},...,x_{n})M)<+\infty $) (see \cite[p.
185]{brunsherzog}). We define the filtered graded ring $\mathrm{gr}%
_{I}R=\oplus _{i=0}^{+\infty }I^{i}/I^{i+1}$, and the filtered graded module 
$\mathrm{gr}_{I}M=\oplus _{i=0}^{+\infty }I^{i}M/I^{i+1}M$, where $I_{0}=R$.
Then, $\mathrm{gr}_{I}R$ is in a natural way a graded ring (here $R/I$ is
Artinian, because after reducing to the case $\mathrm{Ann}_{R}M=0$, it is
easy to see that $\ell (M/IM)<+\infty $ if and only if $\mathrm{rad}I=m$).
Therefore we can define the multiplicity of $M$ on $I$, $e(I,M):=e(\mathrm{gr%
}_{I}M)$ (see \cite[p. 180]{brunsherzog}) and the multiplicity of $R$, $%
e(R):=e(m,R)$ (see \cite[p. 108]{matcomrintheo}). In particular, we can
define the multiplicity of $M$ on $I=(x_{1},...,x_{n})$, where $%
x_{1},...,x_{n}\in R$ is a \emph{system of parameters} of $M$, i.e. $n=%
\mathrm{dim}M$ and $M$ is Artinian, i.e. satisfied the descending chain
condition for submodules (see \cite[p. 74]{atimac}). Besides, under the
former hypothesis and assuming that $x=x_{1},\cdots ,x_{n}$ is a
multiplicity system of $M$, we can define the \emph{Euler Characteristic} as 
\begin{equation*}
\chi (\underline{x},M)=\sum_{i=0}^{n}(-1)^{i}\ell (H_{i}(\underline{x},M)).
\end{equation*}%
For a more technical reformulation of this notion due to Auslander and
Buchsbaum, see \cite{brunsherzog}. Now, a theorem of Serre (see \cite[%
Theorem 4.6.6.]{brunsherzog}) states that $\chi (\underline{x},M)=e(I,M)$,
if $\{x_{1},...,x_{N}\}\subseteq R$ is a system of parameters for $M,$ and
zero otherwise. In particular, if $\{x_{1},...,x_{n}\}\subseteq R$ is a
system of parameters for $M,$ and $\mathrm{dim}M>0$, then $\chi (\underline{x%
},M)=e(I,M)=e(\mathrm{gr}_{I}M)>0$, because $\mathrm{dim}M=\mathrm{dim}(%
\mathrm{gr}_{I}M)$ (see \cite[Theorem 4.4.6.]{brunsherzog}).

\textbf{Socle-Parameters Conjecture, Strong Form (SPCS)}. Let $(T,\eta)$ be
a Gorenstein local ring of dimension $d$. Let $\{x_1,...,x_d\}\subseteq T$
be a system of parameters and write $Q=(x_1,...,x_d)$. Let $u\in T$ be any
lifting of a socle element in $T/Q$, i.e. $\mathrm{Ann}_{T/Q}(\bar{\eta})=(%
\bar{u})$. Let $z\in T$ be a zero divisor. Then $u\cdot z\in Q\cdot (z)$.
This is equivalent to saying that $\ell(H_0(\underline{x},T/(z)))-\ell(H_1(%
\underline{x},T/(z)))>0$.

Now, we prove the last equivalence:

\begin{proposition}
In the situation of the SPCS the following are equivalent.

\begin{enumerate}
\item $u\cdot z\in Q\cdot (z)$.

\item $\mathrm{Ann}_T(z)\nsubseteq Q$.

\item $\ell(H_0(\underline{x},T/(z)))-\ell(H_1(\underline{x},T/(z)))>0$.
\end{enumerate}
\end{proposition}

\begin{proof}
$(1)\Rightarrow (2)$ Consider the following natural short exact sequence  
\begin{equation*}
0\longrightarrow \mathrm{Ann}_T(z)\longrightarrow T \longrightarrow
(z)\longrightarrow0.
\end{equation*}
We know that $T/\mathrm{Ann}_T(z)\cong(z)$, by the isomorphism sending $%
\overline{t}$ to $tz$. After tensoring with $T/Q$ we get 
\begin{equation*}
(z)/(Q(z))\cong T/Q\otimes (z)\cong T/Q\otimes T/\mathrm{Ann}_T(z) \cong T/(%
\mathrm{Ann}_T(z)+Q)).
\end{equation*}
Now, $uz\in Q\cdot (z)$ if and only if $\mathrm{Ann}_T(z)\nsubseteq Q$.
Effectively, $uz\in Q\cdot (z)$ is equivalent to $\overline{uz}=0\in
(z)/(Q(z))$, and it is equivalent to $\overline{u}=0\in T/(\mathrm{Ann}%
_T(z)+Q)$, under the last isomorphism. Therefore, there exists $w\in \mathrm{%
Ann}_T(z)$ and $q\in Q$ such that $u=w+q$, and so $w=u-q\notin Q$ because $%
u\notin Q$. Then $\mathrm{Ann}_T(z)\nsubseteq Q$.

$(2)\Rightarrow (1)$ If $\mathrm{Ann}_{T}(z)\nsubseteq Q$, by Proposition %
\ref{7}, $\overline{u}\in \overline{\mathrm{Ann}_{T}(z)}\subseteq T/Q$. Then
there exists $w\in \mathrm{Ann}_{T}(z)$ such that $\overline{u}=\overline{w}$%
, which means that there is a $q\in Q$ such that $u=w+q$. So, 
\begin{equation*}
uz=(w+q)z=wz+qz=qz\in Q\cdot (z).
\end{equation*}%
$(2)\Leftrightarrow (3)$ $\mathrm{Ann}_{T}(z)\nsubseteq Q$ if and only if $%
Q\varsubsetneq \mathrm{Ann}_{T}(z)+Q$ if and only if $\ell (T/(\mathrm{Ann}%
_{T}(z)+Q))<\ell (T/Q)$. Besides, we have the natural short exact sequence 
\begin{equation*}
0\longrightarrow (z)\longrightarrow T\longrightarrow T/(z)\longrightarrow 0.
\end{equation*}%
Then, after considering the induced long exact sequence for Tor, and noting
that $\mathrm{Tor}_{1}^{T}(T,T/Q)=0$, because $T$ is a $T-$free module, and
therefore flat (see previous results), we get the following exact sequence 
\begin{equation*}
0\longrightarrow \mathrm{Tor}_{1}^{T}(T/(z),T/Q)\longrightarrow
(z)/Q(z)\longrightarrow T/Q\longrightarrow T/(Q+(z))\longrightarrow 0.
\end{equation*}%
Now, since $Q$ is generated by a system of parameters, the $T-$modules $%
T/(Q+(z))$, $T/Q$ and $T/(\mathrm{Ann}_{T}(z))\cong (z)/Q(z)$ are Noetherian
rings of dimension zero and therefore Artinian. In particular, they have
finite length as $T-$modules. Then the submodule $\mathrm{Tor}%
_{1}^{T}(T/(z),T/Q))$ has finite length too. By the additivity of $\ell (-)$%
, we have 
\begin{equation*}
\ell (\mathrm{Tor}_{1}^{T}(T/(z),T/Q))-\ell (T/(\mathrm{Ann}_{T}(z)+Q))+\ell
(T/Q)-\ell (T/(Q+(z))=0.
\end{equation*}%
Hence, 
\begin{equation*}
\ell (T/Q)-\ell (T/(\mathrm{Ann}_{T}(z)+Q))=\ell (T/(Q+(z))-\ell (\mathrm{Tor}%
_{1}^{T}(T/(z),T/Q)).
\end{equation*}%
Then, $\mathrm{Ann}_{T}(z)\nsubseteq Q$ if and only if $\ell (T/(Q+(z))-\ell
(\mathrm{Tor}_{1}^{T}(T/(z),T/Q))>0.$ But $T$ is C-M and then $%
\{x_{1},...,x_{n}\}$ is a regular sequence. Hence, the Koszul Complex is a
free (and then projective) resolution of $T/Q$: 
\begin{equation*}
\dots \longrightarrow K_{2}\longrightarrow K_{1}\longrightarrow
K_{0}\longrightarrow T/Q\longrightarrow 0.
\end{equation*}%
Hence, after tensoring this resolution with $T/(z)$, and taking homology, we
find that $H_{1}(\underline{x},T/(z))\cong \mathrm{Tor}_{1}^{T}(T/Q,T/(z))$
and $H_{0}(\underline{x},T/(z))\cong (T/(z))/\overline{Q})\cong T/((z)+Q)$.
In conclusion, $\mathrm{Ann}_{T}(z)\nsubseteq Q$ is equivalent to the
condition 
\begin{equation*}
\ell (H_{0}(\underline{x},T/(z))-\ell (H_{1}(\underline{x},T/(z))>0.
\end{equation*}
\end{proof}

\begin{proposition}
Assume the same hypothesis as in SPCS and, in addition, that $\mathrm{depth}(%
\overline{\eta},T/(z))\geq d-1$, then SPCS holds.
\end{proposition}

\begin{proof}
Write $a=\mathrm{depth}(\overline{\eta}, T/(z))$. It is a well known fact
that if $H_r(\underline{x}, T/(z))$ denotes the Koszul homology and $q=%
\mathrm{sup}\{r:H_r(\underline{x},T/(z))\neq0\}$ then $a=d-q$,  therefore $%
q=d-a\leq1$. In the case that $d=0$, $Q=0$ and $\mathrm{dim}_T(\eta)=1$,
thus for any element $u\in T$ holds $uz=0\in Q\cdot(z)$. Then  assume $d\geq1
$. Besides, $\{x_1,...,x_n\}\subseteq T$ is a system of parameters for the $%
T-$module $T/(z)$, because $\mathrm{dim}T=\mathrm{dim}T/(z)$ and $\mathrm{dim%
}((T/(z))/(x_1,....,x_d)T/(z))=0$. So, $(T/(z))/(x_1,....,x_d)T/(z)$ is an
Artinian ring. Hence, by previous results, we get 
\begin{equation*}
\ell(H_0(\underline{x},T/(z))-\ell(H_1(\underline{x},T/(z))=\chi(\underline{x%
},T/(z))=e(\overline{Q},T/(z)).
\end{equation*}
Now, by previous comments and the fact that $d\geq1$ we see that $e(%
\overline{Q},T/(z))>0$. \newline
\indent
\end{proof}

\textbf{Socle-Parameters Conjecture, weak Form (SPCW)}. Let $(T,\eta )$ be a
Gorenstein local ring of dimension $d$. Let $\{x_{1},...,x_{d}\}$ be any
system of parameters (if $T$ is mixed characteristic ($\mathrm{char}T/\eta
=p>0$) we assume that $x_{1}=p$). Let $Q=(x_{1},...,x_{d})$, and let $u\in T$
be any lifting of a socle element in $T/Q$, i.e. $\mathrm{Ann}_{T/Q}(\bar{%
\eta})=(\bar{u})$. Let $z$ be a zero divisor. Then $u\cdot z\in Q\cdot (z)$,
which is equivalent to the inequality $\ell (H_{0}(\underline{x}%
,T/(z)))-\ell (H_{1}(\underline{x},T/(z)))>0$.

Note that between the two forms of the SPC the only difference is the fact
that in the mixed characteristic case one can assume that $x_{1}=p$. One
needs this last condition in order to apply Cohen's structure theorem in
mixed characteristic.

\begin{remark}
\label{13}  For proving any of the two versions of the SPC it is enough to
assume that $(T,\eta)$ is complete.
\end{remark}

\begin{proof}
Let $\tau:T\rightarrow \widehat{T}$ be the natural homomorphism to the
completion. Then $\tau$ is an faithfully flat extension and $I\widehat{T}%
\cap T=I$ for any ideal $I$ of $T$ (see \cite[p. 63]{matcomrintheo} ).
Besides, another elementary consequence of faithfully-flatness is that for
any ideals $I,J$ of $T$, $(J:I)\widehat{T}=(J\widehat{T}:I\widehat{T})$.
Now, assume by contradiction that there exists a system of parameters $%
\{x_1,...,x_d\}$ (for the SPCW assume $x_1=p$), a zero divisor $z\in T$ and $%
u\in T$ a lifting of a socle element for $T/Q$ such that $uz\notin Q(z)$.
Let us write $\tau(y)=y^\prime$. Then $\{x_1^\prime,...,x_d^\prime
\}\subseteq \widehat{T}$ is a system of parameters in $\widehat T$, because $%
\hat{\eta}=\eta \widehat{T} =\mathrm{rad}((x_1,...,x_d)\widehat{T})\subseteq 
\mathrm{rad}(x_1^\prime,...,x_d^\prime)\subseteq \hat{\eta}$. Besides, $%
(Q\cdot(z):z)\widehat{T}=(Q\widehat{T}\cdot (z^\prime):z^\prime)$; $\widehat{%
T}/Q\widehat{T}\cong(\widehat{T/Q})=T/Q$, due to the fact that $\overline{%
\eta}^n =(0)$ for some $n>0$; therefore $(u^\prime)=\mathrm{Ann}_{\widehat{T}%
/Q\widehat{T}}(\hat{\eta})$  and so $u^\prime$ is a socle element. Note that 
$p=\mathrm{char}(\widehat{T}/Q\widehat{T})=\mathrm{char}(T/Q)$. Furthermore, 
$\widehat{T}$ is also a Gorenstein ring (see \cite[Theorem 18.3]%
{matcomrintheo} ). \newline
\indent
Finally, $((Q\widehat{T}:z^\prime)+(u^\prime))/(Q\widehat{T}:z)\cong
((Q:z)+u))/(Q:z))\otimes \widehat{T}\neq0$, because $(Q:z)+u)/(Q:z)\neq0$,
since $uz\notin Q(z)$. Then $u^\prime z^\prime\notin Q\widehat{T}\cdot
(z^\prime)$ which contradicts SPC in the complete case.
\end{proof}

\section{The equivalence to the DSC}

First, we review the notion of a \emph{coefficient ring}: If $(R,m,k)$ is
equicharacteristic, a \emph{coefficient field} is a field $K_{0}\subseteq R$
such that the natural projection $\pi :K_{0}\subseteq R\hookrightarrow R/m=k$
is an isomorphism. On the other hand, let $(R,m,k)$ be a complete
quasi-local (that means with a unique maximal ideal $m$ but not necessarily
Noetherian), mixed characteristic, and \emph{separated} ring (i.e. $\cap
_{n\in \mathbb{N}}m^{n}=(0)$). Then a \emph{coefficient ring} for $R$ is a
sub-ring $(D,\eta )\hookrightarrow R$ such that it is a local complete
discrete valuation ring such that $m\cap D=\eta $, and so that the inclusion
induces an isomorphism $D/\eta \cong R/m$. It is elementary to see that if $%
\mathrm{char}R=0$, and $\mathrm{char}k=p>0$, then $D$ is a domain, and
therefore one dimensional, that means exactly that $(D,\eta )$ is a discrete
valuation domain (DVD). A theorem of I. S. Cohen states that for complete
local rings there always exists a coefficient ring (see \cite[Theorem 9,
Theorem 11]{cohenstructure} and \cite[p. 24]{hochsternotes}). In fact, in
the mixed characteristic case ($\mathrm{char}k=p>0$) there exists
coefficient rings which are DVD-s $(D,pD,k)$, with local parameter $p$, or $%
D/p^{m}D$, when $\mathrm{char}R=p^{m}$ (see \cite[Corollary p. 24]%
{hochsternotes}).

\begin{theorem}
\label{14}  The Socle-parameters Conjecture (weak Form) is equivalent to the
Direct Summand Conjecture.
\end{theorem}

\begin{proof}
$SPCW\Rightarrow DSC$. By previous comments we may assume that $%
R\hookrightarrow S$ is a finite extension and $R$ is a unramified regular
local mixed characteristic ring ($\mathrm{char}(R/m)=p>0$) with
algebraically closed residue field $k$. Besides, by Theorem \ref{5} we can
assume that $S=T/(J)$ such that $T=R[y_1,...,y_r]/(f_1(y_1),...,f_r(y_r));$ $%
f_i(y_i)=y_i^{n_i}+a_{i1}y_i^{n_i-1}+\cdots+a_{in_i}$, where $a_{ik}\in m$,
for all indexes $i$,$k$, and $\mathrm{ht}(J)=0$. Now, $T$ is a Gorenstein
local ring by Proposition \ref{6}. Moreover, by Proposition \ref{11} $%
R\hookrightarrow S$ splits if and only if $R_1:=R[x_1,...,x_d]_{m_1}%
\hookrightarrow S_1=T_1/(z)$ splits, where $m_1=m+(x_1,...,x_d)$, $%
T_1=R_1[y_1,...,y_r] /(f_1(y_1),...,f_r(y_r))$ and $z\in T_1$. \newline
\indent
Besides, by previous results $R[x_1,...,x_d]$ is a regular ring. In
particular $R_1$ is regular. Furthermore, $R_1$ is unramified, otherwise
there exist elements $a_i,b_i\in m_1$ and $s\in R_1\smallsetminus m_1$ such
that $sp=\sum_{i=1}^{c}a_ib_i$, and then evaluating in $(0,...,0)$ we get $%
p=s(\underline{0})^{-1}\sum_{i=1}^{c}a_i(\underline{0}) B_i(\underline{0})$
where, $a_i(\underline{0}),b_i(\underline{0})\in m$ and $s(\underline{0})\in
R\smallsetminus m$, so $p\in m^2$, which is a contradiction. \newline
\indent
Hence, $p\notin m^2$ and then $\overline{p}\neq0\in m/m^2$ is a part of a
basis of $m/m^2$ as k-vector space, which is equivalent by the Lemma of
Nakayama to the fact that $p$ is a part of a minimal set of generators of $m$%
, say, $\{w_1=p,...,w_n\}\subseteq m$. Now, $\{w_1=p,...,w_n\}\subseteq T_1$
is a system of parameters in $T_1$, because $\mathrm{dim}(T_1/m_1T_1)= 
\mathrm{dim}(r_1/m_1)=0$ and $\mathrm{dim}T_1=\mathrm{dim}R_1$, since $%
R_1\hookrightarrow S_1$ is a finite extension. \newline
\indent
On the other hand, $z$ is a zero divisor en $T_1$ because $\mathrm{dim}%
T_1/(z)=\mathrm{dim}S_1=\mathrm{dim}R_1=\mathrm{dim}T_1$ and therefore $z$
is contained in a minimal prime of $T_1$, since $T$ is C-M. \newline
\indent
Since $T_1$ is Gorenstein, choose $u\in T_1$ such that $(\overline{u})=%
\mathrm{Ann}_{T_1/m_1T_1}(\overline{\eta})$. By SPCW, $uz\in m_1T_1\cdot(z)$%
, so there exists $a\in m_1T_1$ such that $uz=az$, hence $(u-a)z=0$. But $%
u-a\notin m_1T_1$, because $u\notin m_1T_1$. Therefore $\mathrm{Ann}%
_{T_1}(z)\nsubseteq m_1T_1$, and then, by Theorem \ref{10}, $%
R_1\hookrightarrow S_1$ splits.

$DSC\Rightarrow SPCW$. Let $(T,\eta)$ be a Gorenstein local ring and $%
\{x_1,...,x_d\}\subseteq T$ a system of parameters ($x_1=p$ in the mixed
characteristic case). By Remark \ref{13} we can assume that $T$ is local.
Let $D$ be a coefficient ring for $T$  (which always exits for any complete
local ring). Then, due to the Cohen's Structure Theorem  (see \cite[Lemma 16]%
{cohenstructure}), the ring generated as $D$-algebra by the parameters $%
R=D[x_1,...,x_d]$ is a complete regular local ring with maximal ideal $%
Q=(x_1,...,x_d)$ such that the extension $R\hookrightarrow T$ is finite.
Since $R$ is regular, then, by Serre's theorem (see \cite{matcomrintheo}
Theorem 19.2) $pd_R(T)$ is finite. Hence, for the Auslander-Buchsbaum
formula and the fact that $\mathrm{depth}(Q,T)=\mathrm{depth}(\eta,T)$ (see 
\cite[Exercise 1.2.26]{brunsherzog}), we know that 
\begin{equation*}
\mathrm{pd}_R(T)=\mathrm{depth}(Q,R)-\mathrm{depth}(Q,T)=
\end{equation*}
\begin{equation*}
\mathrm{dim}R-\mathrm{depth}(\eta,T)=d-\mathrm{dim}T=d-d=0.
\end{equation*}
So $T$ is a free $R$-module. Furthermore, $z$ is contained in an associated
prime of $T$ because it is a zero divisor. Since $T$ is C-M, then, by
previous comments, any associated prime is, in fact, a minimal prime. Thus, $%
z$ is contained in a minimal prime $P\in \mathrm{Spec} T$. Moreover, since $T$
is C-M, $T$ is equidimensional, that means, in particular, that $\mathrm{dim}%
T/(z)\geq \mathrm{dim}T/P=\mathrm{dim}T$. In conclusion, 
\begin{equation*}
\mathrm{dim}R/((z)\cap R)=\mathrm{dim}T/(z)=\mathrm{dim}T=\mathrm{dim}R,
\end{equation*}
so $(z)\cap R=(0)$, because $R$ is a domain ($R$ is regular!). Now, to see
that $uz\in Q\cdot (z)$, it is enough to see that $\mathrm{Ann}%
_T(z)\nsubseteq Q$. In fact, if $\mathrm{Ann}_T(z)\nsubseteq Q$ then $%
\overline{\mathrm{Ann}_T(z)}\neq 0$ in $T/Q$. Therefore, by Proposition \ref%
{7}, $\overline{u}\in \overline{\mathrm{Ann}_T(z)}$, there exists $w\in 
\mathrm{Ann}_T(z)$ such that $u-w\in Q$. Thus, $uz=(u-w)z\in Q\cdot (z)$,
because $wz=0$. By Theorem \ref{10}, $R\hookrightarrow T/(z)$ splits if and
only if $\mathrm{Ann}_T(z)\nsubseteq mT= Q$, where $m=(x_1,...,x_d)\subseteq
R$. Hence, the DSC for $R\hookrightarrow T/(z)$ implies $\mathrm{Ann}%
_T(z)\nsubseteq QT$,  and then $uz\in QT\cdot (z)$
\end{proof}

\section{The SPCS for small multiplicities}

A natural way one could attack the SPCS would be by induction on $e(T)$, the
multiplicity of $T$. We notice first that by Remark \ref{13} one may assume
that $T$ is complete. In the case $e(T)=1,$ then, since $T$ is a complete
C-M ring, it is equicharacteristic and therefore unmixed. Hence, by the 
\emph{Criterion for multiplicity one} (see \cite{novak}) $T$ must be a
regular local ring; in particular, this ring is an integral domain, which
implies $z=0$, from which the SPCS follows directly. 

Suppose now that $e(T)=2$. Since $T$ is C-M, it must satisfy the $S2$
condition of Serre: for any $P\in \mathrm{Spec}T$, $\mathrm{depth}(P,T)\geq 
\mathrm{min}(2,\mathrm{dim}T_{P})$, due to the fact that $\mathrm{depth}%
(P,T)=\mathrm{dim}(T_{P})$. Hence, by a Theorem of Ikeda (see \cite[%
Corollary 1.3.]{hunekeremarkmulti}) $T$ is an hypersurface of the form $B/(f)
$, where $B$ is a complete regular local ring. Now, we will prove a more
general result, namely, that the SPCS holds for residue class ring of local
Gorenstein rings which are UFD and C-M, which implies, in particular, the
case of multiplicity two because regular local rings are UFD and C-M (see 
\cite{eisenbud}).

\begin{proposition}
\label{15}  The SPCS holds for Gorenstein rings of the form $T=B/(f)$, where 
$B$ is a local C-M ring which is a UFD and $f\neq0\in B$.
\end{proposition}

\begin{proof}
Let $z\neq0\in T$ be a zero divisor and $\{\overline{y_1},...,\overline{y_d}%
\}\subseteq T$ a system of parameters. we will see  that $\ell(H_0(%
\underline{y},T/(z)))-\ell(H_1(\underline{y},T/(z)))>0$. The minimal prime
ideals of $T$ are just the principal ideals generated by the prime factors
of $f=\prod f_i^{c_i}$, i.e. $P_i=(f_i)$, since $B$ is a UFD. Besides, it is
enough to prove SPCS for $z=\overline{f_i}$, because each zero divisor is a
multiple of one of these, i.e. $z=a\overline{f_i}$ for some $a\in B$, and
thus if $u\overline{f_i}\in Q\cdot(\overline{f_i})$, where $Q=(\overline{y_1}%
,...,\overline{y_d})$ and $(\overline{u})=\mathrm{Ann}_{T/Q}(\overline{\eta})
$ then $uz=ua\overline{f_i}=ua\overline{f_i}\in Q\cdot (a\overline{f_i}%
)=Q\cdot(z)$. Let us fix some $f_j$, then $T/(\overline{f_j})=B/(f_j)$ is a
C-M ring, because is a quotient of a C-M ring by a ideal generated by a
regular element ($B$ is an integral domain) (see \cite[Proposition 18.13]%
{eisenbud}). Since $T$ is equidimensional, $\mathrm{dim}T=\mathrm{dim}T/(f_j)
$ and 
\begin{equation*}
\mathrm{dim}(T/(f_j)/(\overline{y_1},...,\overline{y_d})=\mathrm{dim}%
B/(f_j,y_1,...,y_d)\leq
\end{equation*}
\begin{equation*}
\mathrm{dim}B/(f,y_1,...,y_d)=\mathrm{dim}T/(\overline{y_1},...,\overline{y_d%
})=0,
\end{equation*}
because $(f_j,y_1,...,y_d)\subseteq(f,y_1,...,y_d)\subseteq B$. Hence, $(%
\overline{y})=\{\overline{y_1},...,\overline{y_d}\}\subseteq T/(f_j)$ is a
system of parameters and so it is a regular sequence. Thus $H_1(\underline{y}%
,T/(\overline{f_j}))=0$ (see \cite[Theorem 16.5]{matcomrintheo}). In
conclusion, 
\begin{equation*}
\ell(H_0((\overline{y}),T/(\overline{f_j})))-\ell(H_1((\overline{y}),T/(%
\overline{f_j}))) =\ell(H_0((\overline{y}),T/(\overline{f_j})))
\end{equation*}
\begin{equation*}
=\ell((T/(\overline{f_j})/(\overline{y_1},...,\overline{y_d})))\geq 1>0,
\end{equation*}
because $T/(\overline{f_j})/(\overline{y_1},...,\overline{y_d})\neq(0)$.
This prove the SPCS for $T$.
\end{proof}

Typical examples of local C-M rings which are UFD are localizations of
polynomial rings in prime ideals, or rings of formal power series over DVD
or fields. More generally regular local rings fulfill these conditions (see 
\cite[Corollary 18.7, Theorem 19.19]{eisenbud}).

\section{A New Proof of DSC in the Positive Characteristic Case}

Now, we present a new proof of the DSC in the positive characteristic case
by proving some particular case of the SPCW. The new key ingredient is the
following Lemma. We refer to \cite[Proposition 5.2.6]{robertsmultiplicity}
for a proof.

\begin{lemma}
Let $R$ be a Noetherian ring, $M$ an $R-$module and let $x_{1},...,x_{n}$ be
a sequence of elements in $R$ such that $M/(x_{1},...,x_{n})M$ has finite
length. Let $i\in \{1,...,n\}$. Then, there exists a constant $c$ such that
the length of the Koszul homology module $\ell
(H_{i}(x_{1}^{q},...,x_{n}^{q};M))\leq cq^{n-i}$, for all positive integers $%
q$. 
\end{lemma}

\begin{proposition}
Let $(R,m,k)$ be an equicharacteristic regular local ring, with $\mathrm{char%
}R=p>0$, and $R \hookrightarrow S$ a finite extension. Then  $R
\hookrightarrow S$ splits. 
\end{proposition}

\begin{proof}
After tensoring with the completion of $R$, which is faithfully flat, we can
assume, by previous comments, that $R$ is complete. By Cohen's Structure
Theorem (see \cite[Theorem, p. 26]{hochsternotes}), $R\cong k[[x_1,...,x_n]]$%
, where $\mathrm{char}k=p>0$. Now, we can assume that $k$ is perfect (i.e. $%
k^p=k$), because each extension of the tower $R\hookrightarrow \overline{k}%
\otimes_kR\hookrightarrow \overline{k}[[x_1,...,x_n]]$ is faithfully flat,
where $\overline{k}$ denotes an algebraic closure of $k$. Effectively, $%
R\hookrightarrow \overline{k}\otimes_k R$ is $R-$free and therefore
faithfully flat. Besides, we can identify $\overline{k}\otimes_k R$ with $%
\cup_{i\in I}E_i[[x_1,...,x_n]]$, where $E_i$ runs over all field extensions 
$k\subseteq E_i\subseteq \overline{k}$, such that $[E:k]<+\infty$. From
this, we see that the completion  of the local ring $\overline{k}\otimes_kR$
is exactly $\overline{k}[[x_1,...,x_n]]$ and so  $\overline{k}%
\otimes_kR\hookrightarrow \overline{k}[[x_1,...,x_n]]$ is faithfully flat.

Again, by Theorem \ref{5} and Proposition \ref{6} we can assume that $S\cong
T/J$, where $T$ is a Gorenstein local ring and $\mathrm{ht}\,J=0$. Moreover,
by Proposition \ref{11} and after tensoring with the completion of $%
R_1=R[w_1,...,w_r]_{(m+(w_1,...,w_r))}$, which is isomorphic to $\overline{k}%
[[X_1,...,x_m]]$ (for some $m\geq n$),  we see that $\widehat{R_1}%
\otimes_{R_1}(R_1\otimes_R T)$ has exactly the same form as in Proposition %
\ref{6}. But, now we can assume that $J$ is a principal ideal generated by a
zero divisor. In conclusion, we can assume that $R=k[[x_1,...,x_m]]$, where $%
k$ is a perfect field  and $S=T/(z)$, where $T$ is a Gorenstein local ring
and $z$ is a zero divisor.

Now, we set $P_{q}=(x_{1}^{q},...,x_{n}^{q})\subseteq T$, where $q$ is a
power of $p$. Note that $R^{q}=k[[x_{1}^{q},...,x_{n}^{q}]]\hookrightarrow R$
is finite and thus by the proof of Theorem \ref{14} $R^{q}\hookrightarrow S$
splits if and only if 
\begin{equation*}
\delta =\delta _{R^{q}}(x_{1}^{q},...,x_{n}^{q};T/(z)):=
\end{equation*}%
\begin{equation*}
\ell _{R^{q}}(H_{0}(x_{1}^{q},...,x_{n}^{q};T/(z))-\ell
_{R^{q}}(H_{1}(x_{1}^{q},...,x_{n}^{q};T/(z))>0.
\end{equation*}%
Besides, it is elementary to see that 
\begin{equation*}
\delta =\delta _{R}(x_{1}^{q},...,x_{n}^{q};T/(z)),
\end{equation*}%
because the degree of the residue field extension of $R^{q}\hookrightarrow R$
is one and then

\begin{equation*}
\ell _{R}(H_{i}^{R}(x_{1}^{q},...,x_{n}^{q};T/(z)))=\ell
_{R^{q}}(H_{i}^{R}(x_{1}^{q},...,x_{n}^{q};T/(z)))=
\end{equation*}%
\begin{equation*}
\ell _{R^{q}}(H_{i}^{R^{q}}(x_{1}^{q},...,x_{n}^{q};T/(z))).
\end{equation*}%
The last equality holds because the last two Kozsul homology groups are
isomorph as $R^{q}$-modules.

We will prove that $\mathrm{lim}_{q\rightarrow +\infty }\delta
_{R}(x_{1}^{q},...,x_{n}^{q};T/(z))=+\infty $. In fact, since $%
\{x_{1}^{q},...,x_{n}^{q}\}\subseteq T$ is a system of parameters for the $R-
$module $T/(z)$, we know that 
\begin{equation*}
\chi (x_{1}^{q},...x_{n}^{q};T/(z))=q^{n}\chi (x_{1},...,x_{n};T/(z))
\end{equation*}%
(see Corollary 5.2.4 \cite{robertsmultiplicity}) and by previous comments, 
\begin{equation*}
\chi (x_{1},...,x_{n};T/(z))=e((x_{1},...,x_{n}),T/(z))>0,
\end{equation*}%
because $\mathrm{dim}R>0$ (if $\mathrm{dim}R=0$ then $R$ is a field and the
DSC is trivial). Besides, by the previous corollary we know that there is a
constant $c$ such that $\ell
_{R}(H_{i}(x_{1}^{q},...,x_{n}^{q};T/(z)))<cq^{n-i}$ for each $i=1,...,n$
and each $q$. Combining this we get the following estimate 
\begin{equation*}
\delta
_{R}(x_{1}^{q},...,x_{n}^{q};T/(z))=q^{n}e+\sum_{i=2}^{n}(-1)^{i+1}\ell
_{R}(H_{i}^{R}(x_{1}^{q},...,x_{n}^{q};T/(z))\geq 
\end{equation*}%
\begin{equation*}
q^{n}e-\sum_{i=2}^{n}cq^{n-i},
\end{equation*}%
where $e=e(x_{1}^{q},...,x_{n}^{q};T/(z))>0$. Let's write $%
f(q):=q^{n}e-\sum_{i=2}^{n}cq^{n-i}$. Then, the polynomial $f(q)\rightarrow
+\infty $, when $q\rightarrow +\infty $, because it has positive leading
coefficient. Therefore, $\delta _{R}(x_{1}^{q},...,x_{n}^{q};T/(z))=+\infty $%
. Let's fix $b=p^{h}>0$. Then by the proof of Theorem \ref{14} $%
R^{b}\hookrightarrow T/(z)$ splits. Denote by $\rho _{1}:T/(z)\rightarrow
R^{b}$ a splitting $R^{b}-$homomorphism. Finally, the Frobenius homomorphism 
$F_{b}:R\rightarrow R^{b}$ sending $x\rightarrow x^{b}$ is an isomorphism of
rings. Hence, we can define $\rho :T/(z)\rightarrow R$, by $\rho
(x):=F_{b}^{-1}(\rho _{1}(x^{b}))$. Clearly, $\rho (1)=1$ and if $x\in T/(z)$
and $r\in R$ then 
\begin{equation*}
\rho (rx)=F_{b}^{-1}(\rho _{1}((rx)^{b}))=F_{b}^{-1}(\rho _{1}(r^{b}x^{b}))=
\end{equation*}%
\begin{equation*}
F_{b}^{-1}(r^{b}\rho _{1}(x^{b}))=F_{b}^{-1}(r^{b})F_{b}^{-1}(\rho
_{1}(x^{b}))=r\rho (x).
\end{equation*}%
So $\rho $ is $R-$linear. In view of that $R\hookrightarrow T/(z)$ splits.
\end{proof}
\section*{acknowledgements}
Both authors would like to express their gratitude to the German Academic Exchange Service (DAAD) and to the Universidad Nacional de Colombia for all the support. Danny A. J. G\'omez-Ram\'irez thanks all the members of the family G\'omez Torres for all the support and motivation. Finally, Danny A. J. G\'omez-Ram\'irez was also supported by the Vienna Science and Technology Fund (WWTF) as part of the Vienna Research Group 12-004.



\providecommand{\bysame}{\leavevmode\hbox to3em{\hrulefill}\thinspace} %
\providecommand{\MR}{\relax\ifhmode\unskip\space\fi MR } 
\providecommand{\MRhref}[2]{  \href{http://www.ams.org/mathscinet-getitem?mr=#1}{#2}
} \providecommand{\href}[2]{#2}


\end{document}